\documentclass{siamltexmm}
\usepackage{amsmath}
\usepackage{amssymb}
\usepackage{esint}
\usepackage{color}
\usepackage{verbatim}
\usepackage[english]{babel}

\makeatletter

\newcommand{\todo}[1]{{\color{darkblue}{#1}}}

\definecolor{darkred}{rgb}{.7,0,0}

\definecolor{darkgreen}{rgb}{0.4,0.7,0}

\definecolor{darkblue}{rgb}{0,0,0.7}

\newcommand{\slugmaster}{%
\slugger{suq}{xxxx}{xx}{x}{x--x}}

\usepackage{amsfonts}
\usepackage{tikz-cd}
\usepackage{pgfplots}
\pgfplotsset{compat=newest}

\newtheorem{assumptions}[theorem]{Assumption}

\newtheorem{remark}[theorem]{Remark}

\numberwithin{equation}{section}

\makeatother

\usepackage{babel}
\begin{document}

\title{Long-time Asymptotics of the Filtering Distribution for Partially
Observed Chaotic Dynamical Systems}

\newcommand*\samethanks[1][\value{footnote}]{\footnotemark[#1]}
\author{Daniel Sanz-Alonso 
\thanks{Mathematics Institute, Warwick University, Coventry CV4 7AL, UK.}
\and Andrew M. Stuart  \samethanks
}
\maketitle

\begin{abstract}
The filtering distribution is a time-evolving probability distribution on
the state of a dynamical system, given noisy observations. We
study the large-time asymptotics of this probability distribution for discrete-time, randomly initialized signals that evolve according to 
a deterministic map $\Psi$. The observations are assumed to comprise a low-dimensional
projection of the signal, given by an operator $P$, subject to additive noise. We 
address the question of whether these observations contain sufficient information to 
accurately reconstruct the signal. In a general framework, we establish conditions 
on $\Psi$ and $P$ under which the filtering distributions concentrate around the signal 
in the small-noise, long-time asymptotic regime. Linear systems, the Lorenz '63
and '96 models, and the Navier Stokes equation on a two-dimensional torus are 
within the scope of the theory. Our main findings come as a by-product of computable 
bounds, of independent interest, for suboptimal filters based on new variants of 
the 3DVAR filtering algorithm.  
\end{abstract}


\section{Introduction}
The evolution of many physical systems can be successfully modelled by a deterministic dynamical system for which the initial conditions may not be known exactly. In the presence of chaos, uncertainty in the initial conditions will be
dramatically amplified even in short time-intervals. However, when observations of 
the system are available, they may be used to ameliorate this growth in uncertainty 
and potentially lead to accurate estimates of the state of the system. In this work we provide sufficient conditions on the observations of a wide class of dissipative chaotic differential equations that guarantee long-time accuracy of the estimated state variables. The equations covered by our theory include the Lorenz '63 and '96 models as well as the Navier Stokes equation on a two-dimensional torus.
The importance of these model problems within geophysical applications is highlighted in \cite{majda2006nonlinear},
and their use for testing the efficacy of filtering algorithms is exemplified in \cite{majda2012filtering,law2012evaluating}.

 It is often natural to acknowledge the uncertainty on the initial condition by viewing it as a probability distribution which is propagated by the dynamics. Whenever a new observation of the state variables becomes available, this distribution is updated to incorporate it, reducing uncertainty. This process is performed sequentially in what is known as filtering \cite{crisan2011oxford}. Unfortunately, in almost all situations of applied relevance ---with the exception of finite state signals and the linear Gaussian case--- the analytical expression for these {\em filtering distributions} involves integrals that cannot be computed in closed form. It is thus necessary to employ a numerical algorithm to sequentially approximate the filtering distributions. In order to develop good algorithms a thorough understanding of the properties of these distributions is desirable. The interplay between properties of the filtering distributions and those of their numerical approximations is perhaps best exemplified by the case of filter stability and particle filtering: the long-time behaviour of particle filtering algorithms depends crucially on the sensitivity of filtering distributions to their initial condition \cite{del2004feynman}, \cite{crisan2008stability}, \cite{lei2013convergence}.
The main result of this paper shows long-time concentration of the filtering distributions towards the true underlying signal for partially observed chaotic dynamics. The proofs combine the asymptotic boundedness of a new suboptimal filter with the mean-square optimality of the mean of the filtering distribution as an estimator of the signal  \cite{williams1991probability}. All our examples rely on synchronization properties of dynamical systems. This tool underlies the study of noise-free data assimilation initiated in \cite{hayden2011discrete} for the Lorenz '63 and the Navier Stokes equation. The paper \cite{hayden2011discrete}  motivated studies of the 3DVAR filter (three-dimensional variational method) from meteorology \cite{lorenc1986analysis,parrish1992national} for a variety of dissipative chaotic dynamical systems, conditioned on noisy observations, in \cite{brett2012accuracy} (Navier Stokes), \cite{law2012analysis} (Lorenz '63) and \cite{sanz} (Lorenz '96). Here we study the filtering distribution itself, using modifications of the 3DVAR filter which exploit dissipativity to obtain upper bounds on the true filtering error. We also provide a unified methodology for the analysis. Furthermore, whereas previous work in \cite{brett2012accuracy}, \cite{sanz} required the observation noise to have bounded support, here only finite variance is assumed.

The suboptimal modified 3DVAR filter that we use in our analysis can also be interpreted using ideas from nonlinear observer theory \cite{thau1973observing}, \cite{tarn1976observers}.  Its asymptotic boundedness is proved by a Lyapunov-type argument. Although more sophisticated suboptimal filters could be used to gain insight on the filtering distributions, our choice of modified nonlinear observers is particularly  well-suited to deal with high (possibly infinite) dimensional signals, as indicated by the
fact that the theory includes the Navier-Stokes equation. Filtering in high dimensions is not, in general, well-understood. For example, the question of whether some form of particle filtering could be robust with respect to dimension has received much recent attention \cite{snyder2008obstacles}, \cite{rebeschini2013can}, \cite{beskos2014stability}. By understanding properties of the filtering
distribution in high and infinite dimensions we provide insight that may inform
future development of particle filters.

The paper is organized as follows. In Section  \ref{setupsection} we set up the notation and formulate the questions we address in the rest of the paper. Section \ref{suboptimalfilterssection}  reviews the 3DVAR algorithm from data assimilation and its relation to more general nonlinear observers from the control theory literature. A new truncated nonlinear observer is also introduced. In Section \ref{analysisnonlinearobservers} we prove long-time asymptotic results for these suboptimal filters, and thereby deduce long-time accuracy of the filtering distributions. Section \ref{sec:Application-to-relevant} contains some applications to relevant models and we close in Section \ref{conclusions}.

\section{Set-up}\label{setupsection}

Filtering problems are naturally formulated within the framework of
Hidden Markov Models. The general setting that we consider is that
of a Markov chain $\{v_{j},y_{j}\}_{j\ge0},$ where $\{v_{j}\}_{j\ge0}$
is the {\em signal} process, and $\{y_{j}\}_{j\ge0}$ is the \emph{observation}
process. We assume throughout $y_{0}=0$ so that $y_{0}$ gives no
information on the initial value of the signal and that, for each
$j\ge1,$ $y_{j}$ is a noisy observation of $v_{j}.$ We are interested
in the value of the signal, but have access only to outcomes of the
observation process. We suppose that both take values in a separable
Hilbert space ${\cal H}=({\cal H},\langle\cdot,\cdot\rangle,|\cdot|)$
and that the signal is randomly initialized with distribution $\mu_{0},$
$v_{0}\sim\mu_{0}.$ We assume further that there is a \emph{deterministic}
map $\Psi$ such that

\begin{equation}
v_{j+1}=\Psi(v_{j}),\quad{\text{for}}\, j\ge0,\label{eq:deterministicevolution}
\end{equation}
and therefore all the randomness in the signal comes from its initialization. 

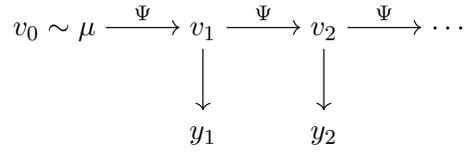
\begin{figure}\begin{center}\begin{tikzcd}  v_0\sim \mu\arrow{r}{\Psi} &v_1\arrow{r}{\Psi}\arrow{d}  &v_2\arrow{r}{\Psi}\arrow{d} &\cdots\\ &y_1&y_2 \end{tikzcd}\end{center}\caption[]{{\em Graphic representation of the dependence structure assumed throughout this paper. Conditional on $v_0,\ldots,v_j,$ the distribution of $v_{j+1}$ is completely determined by $v_j$ via a deterministic map $\Psi;$ therefore the signal process forms a Markov chain. Similarly, conditional on $\{v_j\}_{j\ge 0},$ $\{y_j\}_{j\ge 1}$ is a sequence of independent random variables such that the conditional distribution of $y_j$ depends only on $v_j.$  }}  \end{figure}The
observation process is given by 
\begin{equation}
y_{j}=Pv_{j}+\epsilon w_{j},\quad\text{for }j\ge1,\label{eq:observationmodel}
\end{equation}
where $P$ denotes some linear operator that projects the signal onto
a proper subspace of ${\cal H},$ $\{w_{j}\}_{j\ge1}$ is an i.i.d.
noise sequence ---independent of $v_{0}$--- and $\epsilon>0$ quantifies
the strength of the noise. We define $Q=I-P$.
For mathematical convenience, and contrary
to usual convention, we see both observations and noise as taking
values in the same space ${\cal H}$ as the signal, with the standing
assumptions $Qy_{j}=0,$ $Qw_{1}=0$ and $Pw_{1}=w_{1}$ a.s.%
\footnote{More generally, when an operator $T:{\cal H}\to{\cal H}$ acts on
the observations it should be implicitly understood that $T:{\cal H}\to{\cal H}$
satisfies $T=PT.$ Moreover it will often be assumed that $T\big|_{P}:P{\cal H}\to P{\cal H}$
is positive definite and then the operator $T^{-1}:{\cal H}\to{\cal H}$
should be interpreted as satisfying $PT^{-1}=T^{-1}\big|_{P},\,\,QT^{-1}\equiv0.$ %
} Thus $Q$ is a projection operator onto the unobserved
part of the system. For $j\ge0,$ we let $Y_{j}:=\sigma(y_{i},\, i\le j)$
be the $\sigma$--algebra generated by the observations up to the
discrete time $j.$

Note that the law of $\{v_{j},y_{j}\}_{j\ge0}$ is completely determined
by four elements: the law of $v_{0},$ the map $\Psi,$ the law of
$w_{1}$ and the observation operator $P$. We
will denote by $\mathbb{P}$ the law of $\{v_{j},y_{j}\}_{j\ge0}$
and by $\mathbb{E}$ the corresponding expectation. It will be assumed
throughout that $\mathbb{E}|v_{0}|^{2}<\infty$ and that the observation
noise satisfies $\mathbb{E}w_{1}=0$ and $\mathbb{E}|w_{1}|^{2}<\infty.$
For convenience and without loss of generality we normalize the latter
so that $\mathbb{E}|w_{1}|^{2}=1.$ 

The main object of interest in filtering theory are the conditional
distributions of the signal at discrete time $j\ge1$ given all observations
up to time $j.$ These are known as \emph{filtering distributions}
and will be denoted by 
\[
\mu_{j}(\cdot):=\mathbb{P}\big[v_{j}\in\cdot|Y_{j}\big].
\]

The mean $\widehat{v}_{j}$ of the filtering distribution
$\mu_{j}$ is known as the \emph{optimal filter} 
\[
\widehat{v}_{j}:=\mathbb{E}[v_{j}|Y_{j}]=\int_{\mathcal H} v\, \mu_j(dv).
\]
By the mean-square minimization property of the conditional expectation
\cite{williams1991probability}, this filter is optimal in the sense
that, among all $Y_{j}$-measurable random variables, it is the only
one ---up to equivalence--- that minimizes the $L^{2}$ distance to
the signal $v_{j}$: 
\begin{equation}
\mathbb{E}|v_{j}-\widehat{v}_{j}|^{2}\le\mathbb{E}|v_{j}-z_{j}|^{2},\quad{\rm for\, all\,}Y_{j}{\rm \text{{--measurable}}}\, z_{j}.\label{eq:optimalityproperty}
\end{equation}
In words, $\widehat{v}_{j}$ is the best possible estimator (in the
mean-square sense) of the state of the signal at time $j$ given information
up to time $j.$ The optimal filter is usually, like the filtering
distributions, not analytically available.
 However, by studying suitable
suboptimal filters $\{z_{j}\}_{j\ge0}$ and using (\ref{eq:optimalityproperty}) we can we can find sufficient conditions under which the optimal filter 
is close to the signal in the long-time horizon. We thus provide sufficient 
conditions under which the observations counteract the potentially chaotic 
behaviour of the dynamical system, and allow predictability on infinite 
time-horizons.

The main objective of this paper is to investigate the long-time asymptotic
behaviour of the filtering distribution for discrete-time chaotic
signals, arising from the solution to a dissipative quadratic system
with energy-conserving nonlinearity 
\begin{equation}
\frac{dv}{dt}+Av+B(v,v)=f,\label{generalform}
\end{equation}
which is observed at discrete times $t_{j}=jh,$ $j\ge1,$ $h>0.$
The bilinear form $B(\cdot,\cdot)$ will be assumed throughout to
be symmetric. We denote by $\Psi_{t}$ the one-parameter solution
semigroup associated with $\eqref{generalform}$, i.e. for $v_{0}\in{\cal H},$
$\Psi_{t}(v_{0})$ is the value at time $t$ of the solution to (\ref{generalform})
with initial condition $v_{0}.$ Furthermore we introduce the abbreviation
$\Psi=\Psi_{h}.$ 

Our theory ---developed in Section 4--- relies on two assumptions
that we now state and explain.

\begin{assumptions}\label{Assumption 1}\ \\
1.  ({\bf Absorbing ball property}.) There are constants $r_0, r_1>0$ such that 
\begin{equation}
|\Psi_{t}(v_{0})|^{2}\le\exp(-r_{1}t)|v_{0}|^{2}+r_{0}\big(1-\exp(-r_{1}t)\big),\quad\quad t\ge0.\label{eq:psicontractsoutsideball}
\end{equation}
Therefore, setting $r=\sqrt{2r_0},$ the ball ${\cal B}:=\{u\in{\cal H}: |u|\le r\}$ is absorbing and forward invariant for the dynamical system  \eqref{eq:deterministicevolution}.\\
2. \label{Squeezingpropertyassumption} ({\bf Squeezing property}.) There is a function $V:\,{\cal H}\to [0,\infty)$ such that $V(\cdot)^{1/2}$ is a Hilbert norm equivalent to $|\cdot|,$ a bounded operator $D,$ an absorbing set ${\cal B}_V=\{u\in{\cal H}: V(u)^{1/2}\le R\}\supset \cal B,$ and a constant $\alpha\in(0,1)$ such that, for all $u\in {\cal B},\, v\in {\cal B}_V,$\label{enu:scontraction}
\[
V\Big((I-DP)\big(\Psi(v)-\Psi(u)\big)\Big)\le\alpha V(v-u).
\]
\end{assumptions}

The absorbing ball property  concerns only the signal dynamics. It is satisfied by many dissipative models of the form (\ref{generalform})
---see Section 5. The squeezing property involves both the signal dynamics and the observation
operator $P.$ It is satisfied by several problems of interest
provided that the assimilation time $h$ is sufficiently small and
that the `right' parts of the system are observed; see again Section
5 for examples. We remark that several forms of the squeezing property can be found in the dissipative dynamical systems literature. They all refer to the existence of a contracting part of the dynamics. Their importance for filtering has been explored in \cite{hayden2011discrete}, \cite{brett2012accuracy} and \cite{chueshov2014squeezing}. It also underlies the analysis in \cite{KLS13} and \cite{sanz}, as we make apparent here. We have formulated the squeezing property to suit our analyses and with the intention of highlighting the similar role that it plays to detectability for linear problems, as explained in Subsection \ref{globalresults}. The function $V$ will represent a Lyapunov type function in Section \ref{analysisnonlinearobservers}. 
For all the chaotic examples in Section \ref{sec:Application-to-relevant} the operator $D$ will be chosen as the identity, but other choices are possible. As we shall see, the absorbing ball property is not required when a global form of the squeezing property, as may arise for linear problems, is satisfied.

We will construct suboptimal filters $\{m_j\}_{j\ge 0}$ that are forced to lie in $ {\cal B}_V.$ By the absorbing ball property the signal $v_j$ is contained, for large $j$ and with high probability, in the forward-invariant ball ${\cal B}.$ Therefore, intuitively, the squeezing property can be applied, for large $j,$ to $m_j\in  {\cal B}_V,\, v_j\in  {\cal B}$.

The main result of the paper, Theorem \ref{generaltheorem}, shows
that, when Assumption \ref{Assumption 1} 
holds, the optimal filter accurately tracks the signal. Specifically we show that there is a constant $c>0,$ independent
of the noise strength $\epsilon,$ such that 
\begin{equation}
\limsup_{j\to\infty}\mathbb{E}|v_{j}-\widehat{v}_{j}|^{2}\le c\epsilon^{2}.\label{eq:typicalresult}
\end{equation}
Note that (\ref{eq:typicalresult}) not only guarantees that in the low noise regime
the optimal filter (i.e.$ $ the mean of the filtering distribution)
is ---on average--- close to the signal, but also that the variance
of the filtering distribution is ---on average--- small.
Indeed, since
\[
{\rm var}[v_{j}\big|Y_{j}]=\mathbb{E}\bigg[(v_{j}-\widehat{v}_{j})\otimes(v_{j}-\widehat{v}_{j})\bigg|Y_{j}\bigg]
\]
it follows using linearity of the trace operator that 
\[
{\rm Trace}\,\mathbb{E}\,{\rm var}[v_{j}\big|Y_{j}]=\mathbb{E}|v_{j}-\widehat{v}_{j}|^{2},
\]
and therefore (\ref{eq:typicalresult}) implies
\[
\limsup_{j\to\infty}{\rm Trace}\,\mathbb{E}\,{\rm var}[v_{j}\big|Y_{j}]\le c\epsilon^{2}.
\]
We hence see that (\ref{eq:typicalresult}) guarantees that the variance of the filtering distributions scales as the size of the observation noise, like ${\mathcal O}(\epsilon^2).$ Thus the initial uncertainty in the initial condition which is ${\mathcal O}(1)$ is reduced, in the large-time asymptotic, to uncertainty of ${\mathcal O}(\epsilon)$: the observations have overcome the effect of chaos.

\section{Suboptimal Filters}\label{suboptimalfilterssection}

The aim of this section is to introduce a suboptimal filter, designed to track dynamics satisfying Assumption \ref{Assumption 1}. This filter is based on the 3DVAR algorithm from data assimilation, and nonlinear observers from control applications. We give the necessary background on these in Subsection \ref{3dvarsection} before introducing the new filter in Subsection \ref{truncatednonlinearobserversection}.
 
\subsection{3DVAR Filter}\label{3dvarsection}

The 3DVAR filter approximates the filtering distribution $\mu_{j+1}$
by a Gaussian $N(z_{j+1},C)$  whose mean can be found recursively
starting from a deterministic point $z_{0}\in{\cal H}$ by solving
the variational problem 
\begin{equation}
z_{j+1}:=\text{{argmin}}_{z}\left\{ \frac{1}{2}\left|C_{\sharp}^{-1/2}\bigl(z-\Psi(z_{j})\bigr)\right|^{2}+\frac{1}{2\epsilon^{2}}\left|\Gamma^{-1/2}(y_{j+1}-Pz)\right|^{2}\right\} ,\label{3dvardefinition}
\end{equation}
where $C_{\sharp}$ is a fixed model covariance that represents the
lack of confidence in the model $\Psi,$ and $\Gamma$ is the covariance
operator of the observation noise $w_{1}.$ 

The covariance $C$ of the 3DVAR filter is determined by the Kalman update formula 
\[
C^{-1}=C_{\sharp}^{-1}+P^{T}\Gamma^{-1}P.
\]

It is immediate from (\ref{3dvardefinition}) that $z_{j}$ is $Y_{j}$-measurable
for all $j\ge0,$ and it can be shown \cite{dataassimilationbook}  that the solution
$z_{j+1}$ to this variational problem satisfies 
\begin{equation}
z_{j+1}=(I-KP)\Psi(z_{j})+Ky_{j+1},\label{update3dvar}
\end{equation}
where $K$ is the Kalman gain
\[
K=C_{\sharp}P^{T}(PC_{\sharp}P^{T}+\epsilon^{2}\Gamma)^{-1}.
\]

The 3DVAR filter was introduced, and has been widely applied, in the
meteorological sciences \cite{parrish1992national,lorenc1986analysis}.
Long-time asymptotic stability and accuracy properties ---that guarantee
that the means $z_{j}$ become close to the signal $v_{j}$--- have
recently been studied for the Lorenz '63 model \cite{law2012analysis} subject to additive Gaussian noise, and the Lorenz '96 and Navier-Stokes equation observed subject to bounded noise \cite{sanz}, \cite{brett2012accuracy}. 

 It will be convenient to allow for other choices of operator $K$
in the above definition, and consider the more general recursion 
\begin{equation}
z_{j+1}=(I-DP)\Psi(z_{j})+Dy_{j+1},\label{eq:general3dvar}
\end{equation}
where $D$ is some linear operator that we are free to choose as desired. Filters of the form \eqref{eq:general3dvar} are known as nonlinear observers \cite{thau1973observing}, \cite{tarn1976observers}. 
The 3DVAR filter can be seen as an instance of these where the operator $D$ is is determined by model and noise covariances, and by the observation 
operator.
We now derive a recursive formula for the error made by nonlinear observers when approximating the signal. To that end note, firstly, that the signal $\{v_j\}_{j\ge 0}$ satisfies
$$v_{j+1}=(I-DP)\Psi(v_j) + DP \Psi(v_j).$$
Secondly, using \eqref{eq:observationmodel} at time $j+1,$ combined with the assumption that $Pw_{j+1}=w_{j+1},$ 
$$z_{j+1}=(I-DP)\Psi(z_j)+DP\Psi(v_j)+\epsilon DPw_{j+1}$$
Therefore, substracting the previous two equations, we obtain that the error $\delta_j:=v_j-z_j$ satisfies
\begin{equation}\label{erroreq}
\delta_{j+1} = (I-DP)\bigl(\Psi(v_j)-\Psi(z_j)\bigr) - \epsilon DPw_{j+1}.
\end{equation}
Despite their simplicity nonlinear observers are known to accurately track the signal under 
suitable conditions \cite{tarn1976observers,thau1973observing}.  Equation \eqref{erroreq} plays a central
role in such analysis, and will underlie our analysis too. It demonstrates
the importance of the operator $(I-DP)\Psi$ in the propagation of error;
this operator combines the properties of the dynamical system,
encoded in $\Psi$, with the properties of the observation operator $P$.

\subsection{Nonlinear Observers and Truncated Nonlinear Observers}\label{truncatednonlinearobserversection}
In the remainder of this section we introduce a truncated nonlinear observer 
that is especially tailored to exploit the absorbing ball property
of the underlying dynamics. 

Given a non-empty closed convex subset ${\cal C}\subset{\cal H},$ take $m_{0}\in{\cal C}$
and, for $j\ge0,$ define the \emph{${\cal {\cal C}}$-truncated nonlinear observer} $m_{j+1}$ by 

\begin{equation}
m_{j+1}:=P_{{\cal C}}\Bigl((I-DP)\Psi(m_{j})+Dy_{j+1}\Bigr),\label{eq:truncated3dvardef}
\end{equation}
where $P_{{\cal C}}$ is the orthogonal (with respect to a suitable inner product) projection operator onto the
set ${\cal C}$; this is well-defined for any non-empty closed convex set  \cite{rudin1987real}. In the next section we will analyze the long-time
behaviour of this filter when ${\cal C}$ is chosen as ${\cal B}_V$ and the inner product is the one induced by $V^{1/2}$
(see Assumption \ref{Assumption 1}). 
The main advantage of this truncated filter is that 
$m_{j}\in {\cal B}_V$ for all $j\ge 0,$ and large
uninformative observations $y_{j}$ corresponding to large realizations
of the observation noise $w_{j}$ will not hinder the performance
of the filter. Examples of other truncated stochastic algorithms can be found in \cite{kushner2003stochastic}.

\section{Stochastic Stability of Suboptimal Filters and Filter Accuracy}\label{analysisnonlinearobservers}
In this section we prove long-time accuracy of certain suboptimal filters under different assumptions on the underlying dynamics and observation model. These results are used to establish long-time concentration of the filtering distributions. We start in Subsection \ref{lyapunovmethod} by recalling the Lyapunov method for proving asymptotic boundedness of stochastic algorithms. In Subsection \ref{globalresults} we employ this method to show asymptotic accuracy of nonlinear observers when a global form of the squeezing property is satisfied, as happens for certain linear problems. Finally, in Subsection \ref{analysischaotic} we use truncated nonlinear observers to deal with chaotic models where only the weaker Assumption \ref{Assumption 1} holds. 
\subsection{The Lyapunov Method for Stability of Stochastic Filters}\label{lyapunovmethod}
Consider a Markov chain $\{\delta_j\}_{j\ge 0}$ and think of it as the random sequence of errors made by some filtering procedure. 
The next result, from  \cite{tarn1976observers}, underlies much of the analysis in the following subsections. 
\begin{lemma}
Let $\delta_j^1$ and $\delta_j^2$ be two realizations of the random variable $\delta_j$  and set $\Delta_j=\delta_j^1-\delta_j^2.$ Suppose that there is a function $V$ such that 
\begin{enumerate}
\item $V(0)=0,\quad \quad V(x)\ge \theta|x|^2$ for all $x\in{\cal H}$ and some $\theta>0.$
\item There are real numbers $K>0$ and $\alpha \in (0,1)$ such
that, for all $\Delta_j\in{\cal H},$  
$$\mathbb{E}[V(\Delta_{j+1})|\Delta_j]\le K+\alpha V(\Delta_j).$$
\end{enumerate}
Then 
\begin{equation*}
\theta\mathbb{E}[|\Delta_j|^2|\Delta_0=a]\le\alpha^j V(a)+K\sum_{i=0}^{j-1}\alpha^i.
\end{equation*}
Therefore,
regardless of the initial state $\Delta_0,$ 
\begin{equation*}
\limsup_{j\to\infty}\mathbb{E}|\Delta_j|^2\le \frac{K}{\theta(1-\alpha)}.
\end{equation*}
\end{lemma}
\subsection{Filter Accuracy with Global Squeezing Property}\label{globalresults}

The following results show that if, for some suitable operator $D,$
the map $(I-DP)\Psi$ satisfies a global Lipschitz condition, then it
is possible to use nonlinear observers to deduce long-time accuracy of the filtering
distributions. Although such a global condition does not typically
hold for dissipative chaotic dynamical systems arising in applications,
the following discussion serves as a motivation for the more general
theory in Subsection \ref{analysischaotic}. Moreover, the results in this subsection are
of interest in their own right. In particular they are enough to deal with the important case
of linear signal dynamics. 

\begin{theorem}\label{globallipschitz}
Assume that there is a Hilbert norm $V(\cdot)^{1/2}$ in ${\cal H},$ equivalent to $|\cdot|,$  and a bounded operator $D$ and constant $\alpha\in(0,1)$ such that 
$$V\Bigl((I-DP)\bigl(\Psi(v)-\Psi(u)\bigr)\Bigr)\le \alpha V(v-u) \quad \forall u,v\in{\cal H}. $$
 Define $\{z_{j}\}_{j\ge0}$ by (\ref{eq:general3dvar}). Then
there is a constant $c>0,$ independent of the noise strength $\epsilon,$
such that 
\[
\limsup_{j\to\infty}\mathbb{E}|v_{j}-z_{j}|^{2}\le c\epsilon^{2}.
\]
\end{theorem} 

\begin{proof}
By assumption $V$ satisfies the first condition in Lemma $4.1.$ Set $\delta_j=v_j-z_j.$ Then, using equation \eqref{erroreq} and the independence structure,
\begin{align*}
\mathbb{E}[V(\delta_{j+1})|\delta_j] & =\mathbb{E}\Bigl[V\Bigl((I-DP)\bigl(\Psi(v_{j})-\Psi(z_{j})\bigr)-\epsilon Dw_{j+1}\Bigr)\Big|\delta_j\Bigr]\\
 & =\mathbb{E}\Bigl[V\Bigl((I-DP)\bigl(\Psi(v_{j})-\Psi(z_{j})\bigr)\Bigr)\Big|\delta_j\Bigr]+\epsilon^{2}\mathbb{E}V(Dw_{j+1})\\
 & \le \alpha V(\delta_j)+C\epsilon^{2},
\end{align*}
where $C>0$ is independent of $\epsilon$ and the last inequality follows from the equivalence of norms and the fact that $D$ is bounded. 
Thus the second condition in Lemma $4.1$ holds and the proof is complete.
\end{proof} 

The following corollary is an immediate consequence of the $L^{2}$
optimality property of the optimal filter (\ref{eq:optimalityproperty}).

\begin{corollary}\label{corollarygloballipschitz}
Under the hypothesis of the previous theorem
\[
\limsup_{j\to\infty}\mathbb{E}|v_{j}-\widehat{v}_{j}|^{2}\le c\epsilon^{2}, \quad \quad
\limsup_{j\to\infty}{\rm Trace}\,\mathbb{E}\,{\rm var}[v_{j}\big|Y_{j}]\le c\epsilon^{2}.
\]
\end{corollary}

In the remainder of this subsection we apply, for the sake of motivation, the previous theorem to the case of linear finite dimensional dynamics. We
take ${\cal H}=\mathbb{R}^{d}$ and let the signal be given by 
\begin{equation}
v_{j+1}=Lv_{j},\quad j\ge1,\quad\quad v_{0}\sim\mu_{0}.\label{eq:lineardynamicsrecursion}
\end{equation}
This framework has been widely studied within the control theory community,
mostly ---but not exclusively--- in the case where both the initial
distribution of the signal and the observation noise are Gaussian.
Other than its modelling appeal, this \emph{linear Gaussian} setting
has the exceptional feature that the filtering distributions are themselves
again Gaussian. Moreover,  their means and covariances can be iteratively
computed using the Kalman filter \cite{kalman1960new}. Since the
optimal filter is the mean of the filtering distribution, the explicit
characterization of the Kalman filter yields an explicit characterization
of the optimal filter. Suppose that, for some given
$\widehat{v}_{0}\in\mathbb{R}^{d}$ and $C_{0}\in\mathbb{R}^{d\times d},$
$\mu_{0}=N(\hat{v}_{0},C_{0})$ and suppose further that $w_{1}\sim N(0,\Gamma).$
Then the filtering distributions are Gaussian, $\mu_{j}=N(\widehat{v}_{j},C_{j}),\, j\ge1,$
and the means and covariances satisfy the recursion 
\begin{align}
\widehat{v}_{j+1} & =(I-K_{j+1}P)L\widehat{v}_{j}+K_{j+1}y_{j+1},\nonumber \\
C_{j+1}^{-1} & =C_{j+1|j}^{-1}+\epsilon^{-2}P^{T}\Gamma^{-1}P,\label{eq:covariancekalmanfilter}
\end{align}
where the \emph{predictive Kalman covariance} $C_{j+1|j}$ and Kalman
gain $K_{j+1}$ are given by 
\begin{align*}
C_{j+1|j} & =LC_{j}L^{T},\\
K_{j+1} & =C_{j+1|j}P^{T}(PC_{j+1|j}P^{T}+\epsilon^{2}\Gamma)^{-1}.
\end{align*}
Similar formulae are available when the covariance operator $\Gamma$
is not invertible in the observation space \cite{dataassimilationbook}. 

\begin{remark}
It is clear from $ $(\ref{eq:covariancekalmanfilter})
that the Kalman filter covariance $C_{j}$, which is the covariance
of the filtering distribution $\mu_{j}$, is deterministic and in
particular does not make use of the observations. It follows from
the discussion in Section 2 that in the linear Gaussian setting 
\[
\limsup_{j\to\infty}\mathbb{E}|v_{j}-\widehat{v}_{j}|^{2}\le c\epsilon^{2}
\]
implies
\[
\limsup_{j\to\infty}{\rm Trace}\,C_{j}\le c\epsilon^{2}.
\]
\end{remark}

In the linear setting the global squeezing property in Theorem \ref{globallipschitz} reduces to the control theory notion of detectability, that we now recall. 
\begin{definition}
The pair $(L,P)$ is called \emph{detectable}  if there exists
a matrix $D$ such that $\rho(L-DP)<1,$ where $\rho(\cdot)$ denotes spectral radius. 
\end{definition} 

We remark that the condition $\rho(L-DP)<1$ guarantees the existence of a Hilbert norm in $\mathbb{R}^d$ in which the linear map defined by the matrix $L-DP$ is contractive. It therefore yields a global form of the squeezing property. Note that detectability may hold for unstable dynamics with $\rho(L)>1.$ However the observations need to contain information on the unstable directions. It is not necessary that these are directly observed, but only that we can retrieve information from them by exploiting any rotations present in the dynamics. This is the interpretation of the matrix $D$ in the definition. The next result states the abstract global theorem of the previous section in the setting of linear dynamics. Our aim in including it here is to make apparent the connection between classical control theory \cite{lancaster1995algebraic}, ideas from data assimilation concerning the 3DVAR filter \cite{brett2012accuracy}, \cite{KLS13}, \cite{sanz}, and the new results for chaotic systems observed with unbounded noise in Section \ref{analysischaotic}.

\begin{theorem}Assume that ${\cal H}=\mathbb{R}^{d}$ and $\Psi(v)=Lv$
with $L\in\mathbb{R}^{d\times d}.$ Then if the pair $(L,P)$ is detectable
there is a constant
$c>0$ independent of the noise strength $\epsilon,$ such that
\[
\limsup_{j\to\infty}\mathbb{E}|v_{j}-\widehat{v}_{j}|^{2}\le c\epsilon^{2},
\]
and consequently in the linear Gaussian setting 
\[
\limsup_{j\to\infty}{\rm Trace}\, C_{j}\le c\epsilon^{2}.
\]
\end{theorem} 

\begin{proof}By the Hautus lemma \cite{sontag1998mathematical} the pair
$(L,P)$ is detectable if and only if 
\[
{\rm {Rank}}\begin{pmatrix}\lambda I-L\\
P
\end{pmatrix}=d
\]
for all $\lambda$ with $|\lambda|\ge1$ or, in other words, if ${\rm Ker}(\lambda I-L)\cap{\rm Ker}(P)=\{0\}$
for all $\lambda$ with $|\lambda|\ge1.$ Using this characterization
of detectability it is immediate from the identity 
\[
{\rm Ker}(\lambda I-L)\cap{\rm Ker}(PL)={\rm Ker}(\lambda I-L)\cap{\rm Ker}(P),\quad\lambda\neq0,
\]
that $(L,P)$ is detectable iff $(L,PL)$ is detectable. Now by hypothesis
$(L,P)$ is detectable and so there exists a matrix $D$ such that
$\rho\bigl((I-DP)L\bigr)<1.$ Hence the linear map defined in $\mathbb{R}^{d}$
by the matrix $(I-DP)L$ is globally contractive in some Hilbert norm.
The result follows from Theorem \ref{globallipschitz} and Corollary
\ref{corollarygloballipschitz}.
\end{proof}

\subsection{Filter Accuracy for Chaotic Deterministic Dynamics}\label{analysischaotic}
In this section we study filter accuracy for signals satisfying Assumption \ref{Assumption 1}. Our analysis now makes use of {\em truncated} nonlinear observers \eqref{eq:truncated3dvardef}, which are forced to lie in the absorbing ball ${\cal B}_V$. The idea is that, once the signal gets into the absorbing ball, projecting the filter into ${\cal B}_V$ reduces the distance from the signal, as measured by the Lyapunov function $V.$ This is the content of the following lemma. $P_{{\cal B}_V}x$ denotes the closest point (in the $V^{1/2}$ norm) to $x\in{\cal H}$ in the set ${\cal B}_V.$ Therefore, $P_{{\cal B}_V}x= R^{1/2}\frac{x}{V^{1/2}(x)}$ for $x \notin {\cal B}_V.$ 

\begin{lemma}\label{lemmapb}Let $V^{1/2}(\cdot)$ be a Hilbert norm and let $R>0.$ Set ${\cal B}_V:=\{b\in{\cal H}:V(b)\le R\}$ as in Assumption \ref{Assumption 1}. 
Then, 
\begin{equation}
V(P_{{\cal B}_V}x-b)\le V(x-b),\quad\quad x\in{\cal H},b\in{\cal B}_V.\label{eq:lemmaresultpb}
\end{equation}
\end{lemma}
\begin{proof}
The case $x\in {\cal B}_V$ is obvious so we assume $V(x)>R.$ Let $\langle \cdot,\cdot \rangle_V$ denote the inner product associated with the norm $V^{1/2}.$
 We claim that 
\begin{equation}
\langle P_{{\cal B}_V}x-b,\, x-P_{{\cal B}_V}x\rangle_V\ge0.\label{eq:obtuse angle}
\end{equation}
Indeed we have
\[
\begin{aligned}\langle P_{{\cal B}_V}x-b,\, x-P_{{\cal B}_V}x\rangle_V & =\left\langle R^{1/2}\frac{x}{V^{1/2}(x)}-b,\, x-R^{1/2}\frac{x}{V^{1/2}(x)}\right\rangle_V\\
 & =\left(1-\frac{R^{1/2}}{V^{1/2}(x)}\right)\left\langle R^{1/2}\frac{x}{V^{1/2}(x)}-b,x\right\rangle_V\\
 & =\left(1-\frac{R^{1/2}}{V^{1/2}(x)}\right)\left(R^{1/2}V^{1/2}(x)-\langle b,x\rangle_V\right)\\
 & \ge\left(1-\frac{R^{1/2}}{V^{1/2}(x)}\right)\left(R^{1/2}V^{1/2}(x)-V^{1/2}(b)V^{1/2}(x)\right).
\end{aligned}
\]
Now, $R\ge V(b)$ because $b\in{\cal B}_V$ and the claim is proved.

Finally note that (\ref{eq:obtuse angle}) implies $V(P_{{\cal B}_V}x-b)\le V(x-b).$
To see this recall the elementary fact that for arbitrary $x_{1},x_{2}\in{\cal H}$
we have that $\langle x_{1},x_{2}\rangle_V\ge0$ implies $V(x_{1})\le V(x_{1}+x_{2})$
and choose $x_{1}:=P_{{\cal B}_V}x-b$ and $x_{2}:=x-P_{{\cal B}_V}x$. 
\end{proof}

Using the fact established in Lemma \ref{lemmapb} we are now in a position to prove
positive results about the truncated nonlinear observer, and hence
the optimal filter, in the long-time asymptotic regime.

\begin{theorem}\label{generaltheorem}
Suppose that Assumption \ref{Assumption 1} 
holds. Let $\{m_{j}\}_{j\ge0}$ be the sequence of ${\cal B}_V$-truncated
nonlinear observers given by (\ref{eq:truncated3dvardef}).
Then there is a constant $c>0,$ independent of the noise strength
$\epsilon,$ such that 
\[
\limsup_{j\to\infty}\mathbb{E}|v_{j}-m_{j}|^{2}\le c\epsilon^{2}.
\]
Consequently,

\[
\limsup_{j\to\infty}\mathbb{E}|v_{j}-\widehat{v}_{j}|^{2}\le c\epsilon^{2}, \quad \quad \limsup_{j\to\infty}{\rm Trace}\,\mathbb{E}\,{\rm var}[v_{j}\big|Y_{j}]\le c\epsilon^{2}.
\]
\end{theorem}

\begin{proof}By Lemma \ref{lemmaoutsideb} below, for arbitrary $\delta>0$ there
is $J>0$ such that, for every $j\ge J,$
\begin{equation}
\int_{\{v_{J}\notin {\cal B}\}}V(v_{j}-m_{j})d\mathbb{P}<\delta.\label{eq:notinabsorbingball}
\end{equation}
Now, for $j\ge J$ we have by the absorbing ball property that $v_{J}\in {\cal B}$
implies that $v_{j+1}\in {\cal B},$ and hence by Lemma
\ref{lemmapb} 
\begin{align*}
\int_{\{v_{J}\in {\cal B}\}} & V(v_{j+1}-m_{j+1})d\mathbb{P}\\
& \le\int_{\{v_{J}\in {\cal B}\}}V\Bigl((I-DP)\bigl(\Psi(v_j)-\Psi(z_j)\bigr) - \epsilon Dw_{j+1}\Bigr)d\mathbb{P}\\
 & =\int_{\{v_{J}\in {\cal B}\}}V(\epsilon Dw_{j+1})d\mathbb{P}+\int_{\{v_{J}\in {\cal B}\}}V\Bigl((I-DP)\bigl(\Psi(v_{j})-\Psi(m_{j})\bigr)\Bigr)d\mathbb{P}\\
 & \quad\quad\quad\quad\quad\quad\quad\quad\quad-\int_{\{v_{J}\in {\cal B}\}}\Big\langle\epsilon w_{j+1},(I-DP)\bigl(\Psi(v_{j})-\Psi(m_{j})\bigr)\Big\rangle_V d\mathbb{P}.
\end{align*}
Using the independence structure the last term vanishes, and for the
second term we can employ the squeezing property with $v_{j}\in {\cal B},\: m_{j}\in {\cal B}_V$
to deduce 
\begin{align*}
\int_{\{v_{J}\in {\cal B}\}}V(v_{j+1}-m_{j+1})d\mathbb{P} & \le c\epsilon^{2}+\alpha\int_{\{v_{J}\in {\cal B}\}}V(v_{j}-m_{j})d\mathbb{P}.
\end{align*}
Since $\alpha \in (0,1),$ Gronwall's lemma starting from $J$ gives (for a different constant $c>0$)
\begin{equation}
\limsup_{j\to\infty}\int_{\{v_{J}\in {\cal B}\}}V(v_{j+1}-m_{j+1})d\mathbb{P}\le c\epsilon^{2}.\label{eq:inabsorbingball}
\end{equation}
Finally, combining (\ref{eq:notinabsorbingball}) and (\ref{eq:inabsorbingball})
yields 

\[
\limsup_{j\to\infty}\mathbb{E}{V(v_{j}-m_{j})}\le c\epsilon^{2}+\delta,
\]
and since $\delta>0$ was arbitrary and the norms $V(\cdot)^{1/2}$ and $|\cdot|$ are assumed equivalent the proof is complete. 
\end{proof}

The following lemma is used in the preceding proof.

\begin{lemma}\label{lemmaoutsideb}
 Let $\delta>0.$ Then, with the notation and assumptions of the previous theorem, there is $J=J(\delta)$ such that, for every $j\ge J,$
\[
\int_{\{v_{J}\notin {\cal B}\}}V(v_{j}-m_{j})d\mathbb{P}<\delta.
\]
\end{lemma}

\begin{proof}Firstly, by the assumed equivalence of norms there is
$\theta>0$ such that $V(\cdot)^{1/2}\le \theta|\cdot|.$ Secondly, using the absorbing
ball property it is easy to check that $\mathbb{P}[v_{J}\notin {\cal B}]$
can be made arbitrarily small by choosing $J$ large enough. Therefore,
since we work with the standing assumption that $\mathbb{E}|v_{0}|^{2}<\infty,$ it is possible
to choose $J$ large enough so that 
\begin{align*}
\int_{\{v_{J}\notin {\cal B}\}}\theta^2|v_{0}|^{2}+R^2+2R\theta|v_{0}|d\mathbb{P}\le\delta.
\end{align*}
Then, for $j>J,$ 
\begin{align*}
\int_{\{v_{J}\notin {\cal B}\}}V(v_{j}-m_{j})d\mathbb{P} & \le\int_{\{v_{J}\notin {\cal B}\}}V(v_{j})+V(m_{j})+2V(v_{j})^{1/2}V(m_{j})^{1/2}d\mathbb{P}\\
 & \le\int_{\{v_{J}\notin {\cal B}\}}V(v_{j})+R^2+2RV(v_{j})^{1/2}d\mathbb{P}\\
 & \le\int_{\{v_{J}\notin {\cal B}\}}\theta^2|v_{j}|^{2}+R^2+2R\theta|v_{j}|d\mathbb{P}\\
 & \le\int_{\{v_{J}\notin {\cal B}\}}\theta^2|v_{0}|^{2}+R^2+2R\theta|v_{0}|d\mathbb{P}\le\delta,
\end{align*}
where we used that, for $j>J$ and $v_{J}\notin {\cal B},$
$|v_{j}|\le|v_{0}|$ by (\ref{eq:psicontractsoutsideball}). 
\end{proof}

\section{Application to Relevant Models\label{sec:Application-to-relevant}}

\subsection{Finite Dimensions (Lorenz '63 and '96 Models)}\label{finitedimexamples}

We study first the finite dimensional case ${\cal H}=\mathbb{R}^{d}.$
Our aim is to introduce a general setting for which Assumption \ref{Assumption 1}  holds, and thus the theory of
the previous section can be applied. In order to do so we need to
introduce suitable norms, and some conditions on the general nonlinear
dissipative equation (\ref{generalform}). We start by setting $|\cdot|$
to be the Euclidean norm, and $V(\cdot)=|P\cdot|^{2}+|\cdot|^{2}$. 

Next we introduce a set of hypotheses on the general system (\ref{generalform}),
and the observation matrix $P.$

\begin{assumptions}\label{assumptions6396} \ \\
 1. $\langle Au,u\rangle\ge|u|^{2},\quad\forall u\in{\cal H}.$ \\
 2. $\langle B(u,u),u\rangle=0,\quad\forall u\in{\cal H}.$ (Energy
conserving nonlinearity.)\\
 3. There is $c_{1}>0$ such that $2|\langle B(u,\tilde{u}),\tilde{u}\rangle|\le c_{1}|P\tilde{u}||u||\tilde{u}|,\quad\forall u,\tilde{u}\in{\cal H}.$\\
 4. There is $c_{2}>0$ such that $|B(u,\tilde{u})|\le c_{2}|u||\tilde{u}|,\quad\forall u,\tilde{u}\in{\cal H}.$\\
 5. There are $c_{3}>0$ and $c_{4}\ge0$ such that $\langle Au,Pu\rangle\ge c_{3}|Pu|^{2}-c_{4}|u|^{2}.$
\end{assumptions}

Assumptions \ref{assumptions6396}.1, \ref{assumptions6396}.2 and
\ref{assumptions6396}.4 are satisfied by various important dissipative
equations, including the Lorenz '63  \cite{hayden2011discrete}  (and used in \cite{law2012analysis}),
and Lorenz '96 models \cite{sanz}. Assumptions \ref{assumptions6396}.3
and \ref{assumptions6396}.5 are fulfilled when the `right' parts
of the system are observed. Examples of observation matrices
$P$ that fit into our theory are given ---both for the Lorenz '63
and '96 models--- in Subsections \ref{lorenz63} and \ref{lorenz96}.

The first two items of Assumption \ref{assumptions6396} are enough to ensure the absorbing
ball property Assumption \ref{Assumption 1}. Indeed, if these conditions
hold then taking the inner product of (\ref{generalform}) with $v$
gives
\[
\frac{1}{2}\frac{d}{dt}|v|^{2}+\langle Av,v\rangle+\langle B(v,v),v\rangle=\langle f,v\rangle,
\]
or
\[
\frac{d}{dt}|v|^{2}+|v|^{2}\le|f|^{2}.
\]
Finally, Gronwall's lemma yields Assumption \ref{Assumption 1}.1 with
$r_{0}=|f|^{2}$ and $r_{1}=1,$ and the absorbing ball
\begin{equation}\label{absball6396}
{\cal B}:= \{u\in {\cal H}: |u|\le r:=\sqrt{2}|f| \}.
\end{equation}

We now show that the squeezing property is also satisfied provided
that the time $h$ between observations is sufficiently small. The
proof is based on the analysis of the Lorenz '63 model in \cite{hayden2011discrete}. Recall that $Q=I-P$ is the operator
that projects onto the unobserved part of the system.

\begin{lemma}\label{theorem3dvar} Suppose that Assumption \ref{assumptions6396}
holds and let $r'>0$. Then there is $h^{\star}>0$ with the property
that for all $h<h^{\star},$ $v\in {\cal B},$ and $u\in{\cal H}$
with $|u-v|\le r',$ there exists $\alpha=\alpha(r')\in(0,1)$ such
that 
\[
V\Bigl(Q\bigl(\Psi(v)-\Psi(u)\bigr)\Bigr)\le\alpha V(v-u).
\]
\end{lemma} 
\begin{proof} Denote $\delta_{0}=u-v$ and $\delta(t)=\Psi_{t}(u)-\Psi_{t}(v).$
Lemma \ref{bound} below shows that 
\[
|\delta(t)|^{2}\le b_{1}(t)|\delta_{0}|^{2}+b_{2}(t)|P\delta_{0}|^{2},
\]
where $b_{1}(t)$ and $b_{2}(t)$ are also defined in Lemma \ref{bound}.
Therefore, noting that $V\bigl(Q\delta(t)\bigr)=|Q\delta(t)|^{2}\le|\delta(t)|^{2},$
\begin{align*}
V\bigl(Q\delta(t)\bigr) & \le\max{\{b_{1}(t),b_{2}(t)\}}V(\delta_{0}).
\end{align*}
Since $b_{1}(0)=1,$ $b_{2}(0)=0,$ and $b_{1}'(0)=-1<0$ it follows
that, for all sufficiently small $t,$ $\max{\{b_{1}(t),b_{2}(t)\}}\in (0,1)$
and the lemma is proved.
\end{proof}

The following result has been used in the proof.

\begin{lemma}\label{bound} Suppose that the notation and assumptions of the previous lemma are in force, and that $|\delta_{0}|\le r'.$ Then, for $t\in[0,h),$ 
\begin{align*}
|P\delta(t)|^{2} & \le|P\delta_{0}|^{2}+\left(k_{4}(e^{kt}-1)+k_{5}(e^{2kt}-1)\right)|\delta_{0}|^{2},
\end{align*}
and 
\begin{align*}
 & |\delta(t)|^{2}\le k_{1}(1-e^{-t})|P\delta_{0}|^{2}\\
 & +\left(e^{-t}+k_{2}\left[\frac{e^{kt}-e^{-t}}{k+1}-(1-e^{-t})\right]+k_{3}\left[\frac{e^{2kt}-e^{-t}}{2k+1}-(1-e^{-t})\right]\right)|\delta_{0}|^{2},
\end{align*}
where $k$ and $k_{i},$ $1\le i\le5,$ are constants defined in the proof,
and $k_{3}$ and $k_{5}$ depend on $r'.$ Therefore, 
\begin{equation}\label{pdeltaa1bound}
|P\delta(t)|^{2}\le a_{1}(t)|\delta_{0}|^{2}+|P\delta_{0}|^{2}
\end{equation}
and 
\begin{equation}\label{deltab1b2bound}
|\delta(t)|^{2}\le b_{1}(t)|\delta_{0}|^{2}+b_{2}(t)|P\delta_{0}|^{2},
\end{equation}
where the functions $a_{1},$ $b_{1}$ and $b_{2}$ are defined in
the obvious way from the expressions above.
\end{lemma}
\begin{proof} Firstly,
it is not difficult to check (see for example \cite{KLS13}) that
Assumptions \ref{assumptions6396}.1, \ref{assumptions6396}.2 and
\ref{assumptions6396}.3 imply that there exists a constant $k>0$
such that, for $u\in{\cal H},$ $v\in {\cal B}$ and
$t>0,$
\[
|\delta(t)|^{2}\le e^{kt}|\delta_{0}|^{2}.
\]
Next, using the definition of $\delta$ and the symmetry of $B(\cdot,\cdot)$
it is possible to derive \cite{law2012analysis} the error equation

\begin{equation}
\frac{d\delta}{dt}+A\delta+2B(v,\delta)+B(\delta,\delta)=0.\label{errorinterval}
\end{equation}
Taking the inner product with $\delta$ we obtain 
\[
\frac{1}{2}\frac{d}{dt}|\delta|^{2}+\langle A\delta,\delta\rangle+2\langle B(v,\delta),\delta\rangle=0,
\]
and therefore 
\[
\frac{1}{2}\frac{d}{dt}|\delta|^{2}+|\delta|^{2}\le c_{1}r|\delta||P\delta|\le\frac{1}{2}|\delta|^{2}+\frac12 c_{1}^{2}r^{2}|P\delta|^{2},
\]
i.e. 
\begin{equation}
\frac{d|\delta|^{2}}{dt}+|\delta|^{2}\le c_{1}^{2}r^{2}|P\delta|^{2}.\label{partialdeltabound}
\end{equation}
We now bound $|P\delta|^{2}.$ Taking the inner product of (\ref{errorinterval})
with $P\delta$ 
\[
\frac{1}{2}\frac{d}{dt}|P\delta|^{2}+\langle A\delta,P\delta\rangle+2\langle B(v,\delta),P\delta\rangle+\langle B(\delta,\delta),P\delta\rangle=0.
\]
Hence, 
\begin{align*}
\frac{1}{2}\frac{d}{dt}|P\delta|^{2}+\langle A\delta,P\delta\rangle & \le2|\langle B(v,\delta),P\delta\rangle|+|\langle B(\delta,\delta),P\delta\rangle|\\
 & \le 2c_{2}r|\delta||P\delta|+c_{2}|\delta|^{2}|P\delta|
\end{align*}
and 
\begin{align*}
\frac{1}{2}\frac{d}{dt}|P\delta|^{2}+c_{3}|P\delta|^{2} & \le c_{4}|\delta|^{2}+2c_{2}r|\delta||P\delta|+c_{2}|\delta|^{2}|P\delta|\\
 & \le c_{4}|\delta|^{2}+2c_{2}r|\delta||P\delta|+c_{2}|\delta|e^{kt/2}r'|P\delta|\\
 & \le c_{4}|\delta|^{2}+\frac{2}{c_{3}}c_{2}^{2}r^{2}|\delta|^{2}+\frac{c_{3}}{2}|P\delta|^{2}+\frac{1}{2c_{3}}c_{2}^{2}e^{kt}r'^{2}|\delta|^{2}+\frac{c_{3}}{2}|P\delta|^{2}
\end{align*}
i.e. 
\[
\frac{d}{dt}|P\delta|^{2}\le\left(2c_{4}+\frac{4}{c_{3}}c_{2}^{2}r^{2}+\frac{1}{c_{3}}c_{2}^{2}e^{kt}r'^{2}\right)|\delta|^{2}.
\]
On integrating from $0$ to $t$ and using that $|\delta(t)|^{2}\le|\delta_{0}|^{2}e^{kt}$:
\begin{align*}
|P\delta(t)|^{2} & \le|P\delta_{0}|^{2}+\left(\frac{2c_{4}+\frac{4}{c_{3}}c_{2}^{2}r^{2}}{k}(e^{kt}-1)+\frac{c_{2}^{2}r'^{2}}{2kc_{3}}(e^{2kt}-1)\right)|\delta_{0}|^{2}\\
 & =|P\delta_{0}|^{2}+\left(k_{4}(e^{kt}-1)+k_{5}(e^{2kt}-1)\right)|\delta_{0}|^{2},
\end{align*}
where the last equality defines $k_{4}$ and $k_{5}.$ This proves \eqref{pdeltaa1bound}. Then, going
back to (\ref{partialdeltabound}), 
\begin{align*}
\frac{d}{dt}|\delta|^{2}+|\delta|^{2}\le c_{1}^{2}r^{2}\left\{ |P\delta_{0}|^{2}+\left(k_{4}(e^{kt}-1)+k_{5}(e^{2kt}-1)\right)|\delta_{0}|^{2}\right\} .
\end{align*}
After denoting $k_{1}=c_{1}^{2}r^{2},$ $k_{2}=k_{1}k_{4}$ and
$k_{3}=k_{1}k_{5}$ the inequality above becomes 
\begin{align*}
\frac{d}{dt}|\delta|^{2} & +|\delta|^{2}\le k_{1}|P\delta_{0}|^{2}+\left(k_{2}(e^{kt}-1)+k_{3}(e^{2kt}-1)\right)|\delta_{0}|^{2}.
\end{align*}
Finally, Gronwall's lemma gives \eqref{deltab1b2bound}.
\end{proof} 

The previous lemmas show that Assumption \ref{assumptions6396} implies the squeezing property
Assumption \ref{Squeezingpropertyassumption}.2 provided that the assimilation time $h$ is sufficiently small. Indeed taking 
\begin{equation}\label{bvdef}
{\cal B}_V := \{ u\in {\cal H}: V(u)^{1/2} \le \sqrt{2}r \}
\end{equation}
with $r$ as in \eqref{absball6396}  we have that $|u-v|\le (1+\sqrt{2}) r$ for $u \in {\cal B},\, v\in {\cal B}_V,$ and we are in the setting of Lemma \ref{theorem3dvar} with $r'=(1+\sqrt{2})r$. Moreover, the requirement ${\cal B}\subset {\cal B}_V$ in \ref{Squeezingpropertyassumption}.2 is also fulfilled. Therefore the following result is a direct application of Theorem \ref{generaltheorem}.

\begin{theorem}\label{theorem6396}
Assume that the signal dynamics are defined via a general dissipative
differential equation on $\mathbb{R}^{d}$ with quadratic energy-conserving
nonlinearity of the form (\ref{generalform}), and that Assumption
\ref{assumptions6396} is satisfied.
Then there is $h^{\star}>0$
such that Assumption \ref{Assumption 1} 
is also satisfied for all $h<h^{\star}.$ Therefore, if $\{m_{j}\}_{j\ge0}$
denotes the sequence of ${\cal B}_V$-truncated
nonlinear observers given by (\ref{eq:truncated3dvardef}) and \eqref{bvdef}, then there
is a constant $c>0,$ independent of the noise strength $\epsilon,$
such that, for all discrete assimilation time $h<h^{\star},$ 

\[
\limsup_{j\to\infty}\mathbb{E}|v_{j}-m_{j}|^{2}\le c\epsilon^{2}.
\]
Consequently  
\[
\limsup_{j\to\infty}\mathbb{E}|v_{j}-\widehat{v}_{j}|^{2}\le c\epsilon^{2}, \quad \quad \limsup_{j\to\infty}{\rm Trace}\,\mathbb{E}\,{\rm var}[v_{j}\big|Y_{j}]\le c\epsilon^{2}.
\]
\end{theorem}

\subsubsection{Lorenz '63  Model \label{lorenz63}}

A first example of a system of the form (\ref{generalform})
is the Lorenz '63 model, which corresponds to a three dimensional problem
defined by (\ref{generalform}) with 
\[
A=\left[\begin{array}{ccc}
a & -a & 0\\
a & 1 & 0\\
0 & 0 & b
\end{array}\right],
\]
\[
B(u,\tilde{u})=\left[\begin{array}{c}
0\\
(u_{1}\tilde{u}_{3}+u_{3}\tilde{u}_{1})/2\\
-(u_{1}\tilde{u}_{2}+u_{2}\tilde{u}_{1})/2
\end{array}\right],\quad\quad f=\left[\begin{array}{c}
0\\
0\\
-b(r+a)
\end{array}\right].
\]
 The standard  parameter values are $(a,b,r)=(10,8/3,28).$
Define the projection matrix 
\[
P:=\left[\begin{array}{ccc}
1 & 0 & 0\\
0 & 0 & 0\\
0 & 0 & 0
\end{array}\right].
\]
It is then immediate from the definitions that the first, second and fourth items of Assumption \ref{assumptions6396} are satisfied \cite{hayden2011discrete}. A verification of the third and fifth items can be found in the proof of Theorem 2.5 of \cite{hayden2011discrete}.

To provide insight, in Table 1 we show a Monte Carlo estimate of the mean square error (MSE) made by a truncated nonlinear observer with different values of the observation noise strength $\epsilon.$ The results suggest that the MSE of this suboptimal filter decreases as ${\cal O}(\epsilon^2),$ in agreement with our theoretical analyses. This provides an upper bound for the error made by the optimal filter.

\begin{remark}An accuracy result for the Lorenz '63 model, similar to Theorem \ref{theorem6396} above,
was established in \cite{law2012analysis} using the 3DVAR algorithm.
Indeed truncation is here not needed since a global form
of the squeezing property Assumption \ref{Assumption 1}.2
holds (with $v\in {\cal B},\, u\in{\cal H}$). 
\end{remark}

\begin{table}\label{table1}
\begin{center}
\begin{tabular}{|l|l|}
\hline
$\epsilon$ & MSE \\ \hline
 $1$ & $1.59$\\ \hline
  $0.1$ & $1.3\times 10^{-2}$\\ \hline
  $ 0.01$ & $4.93\times 10^{-4}$\\ \hline
\end{tabular}
\end{center}
\caption{\em Bounds in the MSE given by the truncated nonlinear observer for the Lorenz '63 model. Only the first coordinate is observed. The assimilation time step is $h=0.01$ and the signal was filtered up to time $T=5.$ The MSE was computed using $20$ initializations of $v_0\sim N(0,I);$  for each of these initializations five observation sequences were generated using Gaussian noise. The MSE shown is the Monte Carlo average of the filter error at time $T$ over all these simulations.}
\end{table}

\subsubsection{Lorenz '96 Model \label{lorenz96}}

\label{Lorenz96} Another system that satisfies the assumptions
introduced in this section is the Lorenz '96 model, which is of the
form (\ref{generalform}) with the choices $A=I_{d\times d},$ where
we assume $d\in3\mathbb{N},$ forcing term 
\[
f=\left[\begin{array}{c}
8\\
\vdots\\
8
\end{array}\right],
\]
and bilinear form 
\[
B(u,\tilde{u})=-\frac{1}{2}\left[\begin{array}{c}
\tilde{u}_{2}u_{d}+u_{2}\tilde{u}_{d}-\tilde{u}_{d}u_{d-1}-u_{d}\tilde{u}_{d-1}\\
\vdots\\
\tilde{u}_{i-1}u_{i+1}+u_{i-1}\tilde{u}_{i+1}-\tilde{u}_{i-2}u_{i-1}-u_{i-2}\tilde{u}_{i-1}\\
\vdots\\
\tilde{u}_{d-1}u_{1}+u_{d-1}\tilde{u}_{1}-\tilde{u}_{d-2}u_{d-1}-u_{d-2}\tilde{u}_{d-1}
\end{array}\right].
\]
Define the projection matrix $P$ by replacing every third column
of the identity matrix $I_{d\times d}$ by the zero column vector
\begin{equation}\label{pforlorenz96}
P=\begin{bmatrix}e_{1}, & e_{2}, & 0, & e_{4}, & e_{5}, & 0, & \cdots\end{bmatrix}.
\end{equation}

For a proof that the first, second and fourth items of  Assumption \ref{assumptions6396} are satisfied see Property 2.1.1 in \cite{sanz}. The third item results from combining 
Property 2.1.1 and Property 2.2.2 in \cite{sanz}. Finally, since $A=I$ the fifth item holds with $c_3=1, c_4=0.$

 As for the Lorenz '63 model, we show a Monte Carlo estimate of the error made by a truncated nonlinear observer in Table 2. Again the error decreases as $\epsilon^2.$

\begin{table}\label{table2}
\begin{center}
\begin{tabular}{|l|l|}
\hline
$\epsilon$ & MSE \\ \hline
 $1$ & $1.11$\\ \hline
  $0.1$ & $1.08\times 10^{-2}$\\ \hline
  $ 0.01$ & $3.36\times 10^{-4}$\\ \hline
\end{tabular}
\end{center}
\caption{\em Same experiment as in Table $1,$ now for the Lorenz '96 model with the observation operator \eqref{pforlorenz96}. }
\end{table}

\subsection{Infinite Dimensions (Navier-Stokes Equation)}\label{infinitedimexample}

It is well known \cite{blomker2012accuracy} that the 
incompressible Navier-Stokes
equation on the torus $\mathbb{T}^{2}=[0,l]\times[0,l]$ can be written
in the form (\ref{generalform}) as we now recall.

Let $\hat{H}$ be the space of zero-mean, divergence-free, vector-valued
polynomials $u$ from $\mathbb{T}^{2}$ to $\mathbb{R}^{2}.$ Let
$H$ be the closure of $\hat{H}$ with respect to the $L^{2}$ norm.
Finally, let $P_{H}:(L^{2}(\mathbb{T}^{2}))^{2}\to H$ be the Leray-Helmholtz
orthogonal projector. Then, the operator $A$ and the symmetric bilinear
form $B$ in (\ref{generalform}) are given by 
\[
Au=-\nu P_{H}\Delta,\quad\quad B(u,v)=\frac{1}{2}P_{H}[u\cdot\nabla v]+\frac{1}{2}P_{H}[v\cdot\nabla u],
\]
were $\nu$ is the viscosity. 

We assume that $f\in H$ so that $P_{H}f=f.$ In the periodic case
considered here $A=-\nu\Delta$ with domain ${\cal D}(A)=H^{2}(\mathbb{T}^{2})\cap H.$
Moreover, the solution to the Navier-Stokes
equation (see below for the precise definition) can be written as a Fourier series
\[
v=\sum_{k\in{\cal K}}v_{k}e^{ikx},\quad{\cal K}=\left\{ \frac{2\pi}{L}(n_{1},n_{2}):n_{i}\in\mathbb{Z},(n_{1},n_{2})\neq(0,0)\right\} .
\]
The Fourier coefficients encode the divergence-free property and hence may be written as
$v_k=v_k' k^{\perp}/|k|$ for scalar coefficients $v_k'$.
We now define the observation operator $P=P_{\lambda}$ in the general
observation model (\ref{eq:observationmodel}) as 
\[
P_{\lambda}u=\sum_{k^{2}\le\lambda}u_{k}e^{ikx},
\]
and set $Q_{\lambda}=I-P_{\lambda}.$ Several choices of noise fit
into our theory, and a natural one is given by 
\begin{equation}
w_{1}=\sum_{k^{2}\le\lambda}\xi_{k}e^{ikx},\label{eq:noiseinnsedefinition}
\end{equation}
where $\xi_{k}\sim N\Bigl(0,\bigl(k^{2}n(\lambda)\bigr)^{-1}\Bigr)$ and $n(\lambda):=\#\{k:\, k^{2}\le\lambda\}.$

We let ${\cal H}$ be the closure of $\hat{H}$ with respect to the
$H^{1}$ norm, and define a norm in ${\cal H}$ 
\[
\|u\|_{H^{1}}^{2}:=\sum_{k\in{\cal K}}k^{2}|u_{k}|^{2},
\]
which is equivalent to the $H^{1}$ norm. Note that with this definition
$\mathbb{E}\|w_{1}\|_{H^{1}}^{2}=1.$ 

The following theorem ---see \cite{temam1995navier}, \cite{constantin1988chicago}
or \cite{robinson2001infinite}--- guarantees the existence and uniqueness
of strong solutions to this problem with initial conditions in ${\cal H}.$
\begin{proposition}\label{3dvarnsetheorem} Let $u_{0}\in{\cal H}$ and
$f\in H$. Then (\ref{generalform}) has a unique strong solution
\[
u\in L^{\infty}\bigl((0,T);{\cal H}\bigr)\cap L^{2}\bigl((0,T);{\cal D}(A)\bigr)\quad\text{and}\quad\frac{du}{dt}\in L^{2}\bigl((0,T),H\bigr)
\]
for any $T>0.$ Furthermore, this solution is in $C([0,T];{\cal H})$
and depends continuously on the initial data $u_{0}$ in the ${\cal H}$
norm. \end{proposition}

Take $|\cdot|=V(\cdot)^{1/2}=\|\cdot\|_{H^{\text{1 }}}.$ It is not difficult
to prove the absorbing ball property 
for the Navier-Stokes equation with initial conditions in ${\cal H}$
\cite{robinson2001infinite}. Indeed there is $\theta=\theta(\nu)>0$ such that, for every $u\in {\cal H},$ $|Au|^2\ge \theta |u|^2.$ Then Assumption \ref{Assumption 1}.1 is satisfied with $r_0=|f|^2\theta^2$ and $r_1=\theta.$ We hence set 
\begin{equation}\label{absballnse}
{\cal B}= \{u\in {\cal H}: |u|\le r:=\sqrt{2}\frac{|f|}{\theta} \}.
\end{equation}
The following squeezing property
is taken from \cite{brett2012accuracy}, which uses the analysis in \cite{hayden2011discrete}.

\begin{lemma}\label{Assumption nse}
For every $r'>0$ there are constants $\alpha=\alpha(r')\in(0,1)$ and
$\lambda_{\star}=\lambda_{\star}(r')>0$ with the property that, for
$\lambda>\lambda_{\star},$ there exists $h^{\star}=h^{\star}(r',\,\lambda)$
such that, for all $u,v\in B(r'):=\{x\in{\cal H}:\:V(x)^{1/2}\le r'\},$
and assimilation time $h<h^{\star},$ 
\[
V\Bigl(Q_{\lambda}\bigl(\Psi(v)-\Psi(u)\bigr)\Bigr)\le\alpha V(v-u).
\]
\end{lemma} 

The previous lemma yields Assumption \ref{Assumption 1}.2.  for sufficiently small assimilation time $h$ by choosing ${\cal B}_V = {\cal B}$ and $r'=2r.$
The next result is then a straightforward application of Theorem \ref{generaltheorem}.

\begin{theorem}
Take $|\cdot|$ and $V$ as above, and let $\{m_{j}\}_{j\ge0}$
be the sequence of ${\cal B}_V$-truncated nonlinear observers with ${\cal B}_V = {\cal B}$ given by \eqref{absballnse}. Then there are $h^{\star},\,\lambda_{\star}>0,$ such
that for all $h<h^{\star}$ and $\lambda>\lambda_{\star}$ Assumption
\ref{Assumption 1} is satisfied
and therefore there exists a constant $c>0,$ independent of the noise
strength $\epsilon,$ such that

\[
\limsup_{j\to\infty}\mathbb{E}|v_{j}-m_{j}|^{2}\le c\epsilon^{2},
\]
Consequently, 
\[
\limsup_{j\to\infty}\mathbb{E}|v_{j}-\widehat{v}_{j}|^{2}\le c\epsilon^{2}, \quad \quad
\limsup_{j\to\infty}{\rm Trace}\,\mathbb{E}\,{\rm var}[v_{j}\big|Y_{j}]\le c\epsilon^{2}.
\]

\end{theorem}

\section{Conclusions}\label{conclusions}

We conclude by summarizing our work and highlighting future directions.
\begin{itemize}
\item Noisy observations can be used to compensate, in the long-time asymptotic
regime, for uncertainty in the initial conditions of unstable or chaotic
dynamical systems.
\item We have determined conditions on the dynamics and observations under
which the optimal filter accurately tracks the signal (and the variance
of the filtering distributions becomes small) in the long-time asymptotic. 
\item These properties of the true filtering distribution are potentially
useful for the design of improved algorithmic approximations of the filtering
distributions.
\item We have introduced a modification of the 3DVAR filter as a tool to
prove our results. This new filter is potentially of interest in its
own right as a practical algorithm.
\end{itemize}

\bigskip

\paragraph{Acknowledgments}

The authors are thankful to Alex Beskos, Dan Crisan, Ajay Jasra and Ramon van Handel 
for inspiring discussions.

\bibliography{references}
{} \bibliographystyle{siam.bst} 
\end{document}